\documentclass[11pt]{amsart}

\PassOptionsToPackage{backref=page}{hyperref}
\usepackage{amsthm} 
\usepackage{amssymb}
\usepackage{graphicx}									 
\usepackage{stmaryrd}                  
\usepackage[draft]{pdfcomment}
\usepackage{bbm}
\usepackage[ngerman,british]{babel}
\usepackage[arrow, matrix, curve]{xy}
\usepackage{color}
\usepackage[pagewise]{lineno}

\usepackage{enumitem}

\newtheorem{thm}{Theorem}[section]

\newtheorem{lem}[thm]{Lemma}

\newtheorem{prp}[thm]{Proposition}

\theoremstyle{definition}
\newtheorem{dfn}[thm]{Definition}

\theoremstyle{remark}
\newtheorem{qst}{Question}
\newtheorem{rem}[thm]{Remark}
\newtheorem{exl}[thm]{Example}

\newcommand{\R}{\mathbb{R}}

\newcommand{\Z}{\mathbb{Z}}

\newcommand{\Orb}{\mathcal{O}}

\newcommand{\Or}{\mathrm{O}}

\newcommand{\oB}{U}
\newcommand{\cB}{B}

\newcommand{\To}{\rightarrow}

\title[Orbifolds from a metric viewpoint]{Orbifolds from a metric viewpoint}
\author{Christian Lange}
\address{Christian Lange, Mathematisches Institut der Universit\"at zu K\"oln, Weyertal 86-90, 50931 K\"oln, Germany}
\email{clange@math.uni-koeln.de}

\thanks{The author was partially supported by the DFG funded project SFB/TRR 191.}


\begin{document}

\begin{abstract} We characterize Riemannian orbifolds and their coverings in terms of metric geometry. In particular, we show that the metric double of a Riemannian orbifold along the closure of its codimension one stratum is a Riemannian orbifold and that the natural projection is an orbifold covering.
\end{abstract}

\maketitle

\section{Introduction}

Orbifolds were introduced by Satake in the $50$s under the name of V-manifolds \cite{MR0079769,Satake} and rediscovered by Thurston in the $70$s, when also the term ``orbifold'' was chosen, cf. \cite{MR2883685}. Thurston moreover defined the notion of coverings of orbifolds and showed that the usual theory of coverings and fundamental groups works in the setting of orbifolds \cite{Thurston}. This theory was later generalized to the setting of groupoids by Haefliger \cite{MR0100269,MR0755163} (see also \cite{MR1744486}, and \cite{MR1466622,MR1950948} for the relation between groupoids and orbifolds). Orbifolds can be endowed with a Riemannian metric. In this form they for instance arise as quotients of isometric Lie group actions on Riemannian manifolds \cite{{MR2719410}} or as Gromov-Hausdorff limits \cite{MR1145256}. Our purpose is to characterize such Riemannian orbifolds and their coverings in terms of metric geometry. Besides being interesting on its own right we describe some applications of this perspective. 

Recall that a \emph{length space} (or \emph{intrinsic metric space})  is a metric space in which the distance between any pair of points can be realized as the infimum of the lengths of all rectifiable paths connecting these points \cite{MR1835418}. The following definition of a Riemannian orbifold was proposed to us by Alexander Lytchak.

\begin{dfn}\label{dfn:orbifold} A \emph{Riemannian orbifold} of dimension $n$ is a length space $\Orb$ such that for each point $x \in \Orb$ there exist an open neighborhood $U$ of $x$ in $\Orb$ and a connected Riemannian manifold $M$ of dimension $n$ together with a finite group $G$ of isometries of $M$ such that $U$ and $M/G$ are isometric with respect to the induced length metrics.
\end{dfn}

Here $M$ is endowed with the induced length metric and $M/G$ with the corresponding \emph{quotient metric}, which measures the distance between orbits in $M$. In Section \ref{sec:basics_Rieman_orbi} we explain in which sense this definition is equivalent to the original definition of a Riemannian orbifold.

To obtain a metric notion of orbifold coverings we follow an idea by Lytchak to view submetries with discrete fibers as branched coverings \cite{MR1938523}. We denote the closed balls in a metric space $X$ as $\cB_r(x)$ and the open balls as $\oB_r(x)$. A map $p:X \To Y$ between metric spaces is called \emph{submetry} if $p(\cB_r(x))=\cB_r(p(x))$ holds for all $x\in X$ and all $r\geq 0$. It is known that a submetry between Riemannian manifolds is a Riemannian submersion \cite{MR1800594}. In particular, a submetry between Riemannian manifolds of the same dimension is a local isometry. Using this notion we can characterize \emph{coverings of Riemannian orbifolds} in metric terms as follows.

\begin{thm}\label{thm:covering_prop} For a map $p: \Orb \To Y$ from an $n$-dimensional Riemannian orbifold $\Orb$ to a metric space $Y$ the following properties are equivalent.
\begin{enumerate}
\item $Y$ is an $n$-dimensional Riemannian orbifold and $p$ is a covering of Riemannian orbifolds in the sense of Thurston, cf. Definition \ref{dfn:riem_orb_covering}.
\item $p$ has discrete fibers, is onto, locally $1$-Lipschitz and each point $y \in Y$ has a neighborhood $U$ such that the restriction of $p$ to $p^{-1}(U)$ is a submetry with respect to the restricted metrics.
\end{enumerate}

Moreover, if $\Orb$ is complete, then each of the conditions $(i)$ and $(ii)$ is satisfied if and only if $p$ is a submetry with discrete fibers. Also, in case $(ii)$ the space $Y$ is locally an Alexandrov space. If it is $n$-dimensional then the assumption on the discreteness of the fibers of $p$ is superfluous. 
\end{thm}

In particular, the theorem applies to coverings of Riemannian manifolds. Note that the additional characterization indeed requires the completeness assumption: in Example \ref{exl:covering_but_no_sub} we construct a covering of noncomplete Riemannian manifolds which is not a submetry. Moreover, a map $p: X \To \Orb$ from an Alexandrov space $X$ to a Riemannian orbifold $\Orb$ satisfying the properties stated in $(ii)$ does not need to be a covering of Riemannian orbifolds. For instance, a double covering from a circle of length $4\pi/3$ to a circle of length $2\pi/3$ induces such a map between the corresponding Euclidean cones. The proof of Theorem \ref{thm:covering_prop} works by induction on the dimension and relies on the covering space theory of orbifolds (i.e. on Theorem \ref{thm:existence_universal_orbifold}) as well as on structural results about submetries between Alexandrov spaces due to Lytchak \cite{MR1938523}.
 
A particular example of a Riemannian orbifold covering is given by the quotient map of an isometric, proper action by a discrete group on a Riemannian orbifold. Another canonical construction works as follows. The set of points in a Riemannian orbifold $\Orb$ which have a neighborhood that is isometric to the quotient of a Riemannian manifold by an isometric reflection is called the \emph{codimension 1 stratum} of $\Orb$. Its closure coincides with the boundary of $\Orb$ in the sense of Alexandrov geometry \cite{MR1185284}. The double of $\Orb$ along this closure admits a natural metric with respect to which the two copies are isometrically embedded (see Section \ref{sec:double_covering}). In Section \ref{sec:double_covering} we prove the following statement.

\begin{prp} \label{prp:metric_double} The metric double of a Riemannian orbifold along the closure of its codimension $1$ stratum is a Riemannian orbifold and the natural projection to $\Orb$ is a covering of Riemannian orbifolds.
\end{prp}

Like the proof of Theorem \ref{thm:covering_prop} the proof of Proposition \ref{prp:metric_double} works by induction on the dimension and relies on the covering space theory of orbifolds (in the form of Theorem \ref{thm:existence_universal_orbifold}). 

\subsection{Applications}

Theorem \ref{thm:covering_prop} is applied in a paper by R.~Mendes and M.~Radeschi in which they establish a close connection between submetries with smooth fibers from the unit sphere $\mathbb{S}(V)$ in a finite-dimensional real vector space $V$ and Laplacian subalgebras in the polynomial ring $\mathbb{R}[V]$ \cite{LaplacianSubmetries}. The latter are subalgebras which contain $r:=\sum_i x_i^2$, where $x_i$ is an orthonormal basis of the dual space $V^*$, and are invariant under the action of the Laplace operator $\Delta$ on $\mathbb{R}[V]$. Proposition \ref{prp:metric_double} can moreover be conveniently applied in the solution of the topological question of when the quotient of $\R^n$ by a finite subgroup of the orthogonal group $\Or(n)$ is a topological manifold with boundary \cite{Lange1}. In dimension $2$ it has also been applied in the papers \cite{Lange2} and \cite{Lange3} about closed geodesics on $2$-orbifolds.

\subsection{Structure of the paper} 

In Section \ref{sec:basics_Rieman_orbi} we show the equivalence of Definition \ref{dfn:orbifold} with the usual definition of Riemannian orbifolds. Moreover, we recall basic notions about (Riemannian) orbifolds and their coverings. In Section \ref{sec:metric_cover} we prove Theorem \ref{thm:covering_prop}. The proof in the case where $p$ has finite fibers is simpler and is carried out in the first part of that section. In Section~\ref{sec:double_covering} we prove Proposition \ref{prp:metric_double}. 
\newline
\newline
\emph{Acknowledgements.} I would like to thank Claudio Gorodski, Alexander Lytchak, Ricardo Mendes and Marco Radeschi for useful discussions. I am moreover grateful to the referee for comments that helped to improve the exposition.

The author was partially supported by the DFG funded project SFB/TRR 191. The support is gratefully acknowledged.

\section{Basics on Riemannian orbifolds and their coverings}\label{sec:basics_Rieman_orbi}

\subsection{Definition of Riemannian orbifolds}\label{sub:def_Riem_orbifold}

An $n$-dimensional smooth orbifold in the sense of Thurston is defined in terms of a Hausdorff space $X$ and a collection of data $\{(\tilde U_i, G_i, U_i,p_i)\}_{i\in I}$, a so-called \emph{atlas}, consisting of $n$-dimensional smooth manifolds $\tilde U_i$ on which finite groups $G_i$ act smoothly, an open covering $\{U_i\}_{i\in I}$ of $X$, and projections $p_i : \tilde U_i \To U_i$ that induce homeomorphisms $\overline{p}_i:\tilde U_i/G_i \To U_i$. The \emph{charts} $(\tilde U_i, G_i, U_i,p_i)$ must satisfy the following compatibility condition. For two points $x_i \in \tilde U_i$ and $x_j \in \tilde U_j$ with $p_i(x_i)=p_j(x_j)$ there should exist open neighborhoods $\tilde V_i$ of $x_i$ in $\tilde U_i$ and $\tilde V_j$ of $x_j$ in $\tilde U_j$ and a diffeomorphism $\phi : \tilde V_i \To \tilde V_j$ with $p_i=p_j \circ \phi$ on $\tilde V_i$. Sometimes an equivarience condition on the transition map $\phi$ is demanded, which is however already implied by the condition that $p_i=p_j \circ \phi$ holds on $\tilde V_i$ as can be seen with the help of a compatible Riemannian metric. Since we will also need the respective statement later, let us record it here.

\begin{lem} \label{lem:group_dete_deck_trans} Let $M$ be a Riemannian manifold on which a finite group $G$ acts isometrically and let $p:M \To M/G$ be the quotient map. Then any isometry $\phi:U \To V$ between open, connected subsets of $M$ with $p=p\circ \phi$ on $U$ is the restriction of the action of some element $g\in G$ on $M$.
\end{lem}
\begin{proof} Let $x \in U$ be a regular point for the covering $p:M \To M/G$. Since $p$ is a local isometry in a neighborhood of $x$ there exists some $g\in G$ such that the action of $g$ and $\phi$ coincide in this neighborhood. Now the claim follows from the fact that a Riemannian isometry of a connected space is determined by its local behaviour \cite[Lem.~4.2]{MR1138207}.
\end{proof}

Hence, in order to show that a Riemannian orbifold in the sense of Definition \ref{dfn:orbifold} has a natural structure of a smooth orbifold one only needs to prove the following lemma. Its proof is based on the fact that a finite subgroup $G<\Or(n)$ is determined up to conjugation by the metric quotient $S^{n-1}/G$. This is proven by hand in \cite[Lem.~1]{MR1935486} and can also be deduced from the covering space theory of Riemannian orbifolds, see Remark \ref{rem:clarify_citation}. Recall that a ball in a Riemannian manifold centered at $x$ is called \emph{normal}, if its closure is contained in the diffeomorphic image of an open neighborhood of the origin in the tangent space $T_x M$ under the exponential map.

\begin{lem}\label{lem:metric_det_group} Let $\oB_r(x)\subset M$ and $\oB_r(\bar x)\subset \bar{M}$ be normal balls in $n$-dimensional Riemannian manifolds $M$ and $\bar M$. Suppose finite groups $G$ and $\bar{G}$ act isometrically and effectively on $M$ and $\bar M$ and fix the points $x$ and $\bar x$, respectively. Suppose further that the quotients $\oB_r(x)/G$ and $\oB_r(\bar x)/\bar G$ are isometric. Then there exists an isometry $\phi: \oB_r(x) \To \oB_r(\bar x)$ that conjugates the action of $G$ on $\oB_r(x)$ to the action of $\bar G$ on $\oB_r(\bar x)$.
\end{lem}
\begin{proof} Since the balls $\oB_r(x)$ and $\oB_r(\bar x)$ are normal, the cosets of $x$ and $\bar x$ in $\oB_r(x)/G$ and $\oB_r(\bar x)/\bar G$ are the unique points whose $r$-neighborhoods cover all of $\oB_r(x)/G$ and $\oB_r(\bar x)/\bar G$, respectively. Therefore an isometry between $\oB_r(x)/G$ and $\oB_r(\bar x)/\bar G$ has to map the coset of $x$ to the coset of $\bar x$. It follows that also the spaces of directions at $x$ and $\bar x$, i.e. the quotients of the unit spheres in the tangent spaces at $x$ and $\bar x$ by the linearized actions of $G$ and $\bar G$, are isometric. Hence, by \cite[Lem.~1]{MR1935486} these linearized actions are conjugated by an isometry. Using the exponential map we obtain a diffeomorphism $\phi: \oB_r(x) \To \oB_r(\bar x)$ that conjugates the actions of $G$ and $\bar G$, and that descends to an isometry between the respective quotient spaces.

Since the projections to the quotients are local isometries over the regular part (i.e. the set of points with trivial isotropy), the diffeomorphism $\phi$ is a Riemannian isometry there. Moreover, for any pair of points $x_0,x_1 \in \oB_r(x)$ and any $\varepsilon>0$ one can find a path $\gamma :[0,1] \To \oB_r(x)$ and a subdivision $0=t_0<t_1<\dots<t_k=1$ such that $\gamma(0)=x_0$, $\gamma(1)=x_1$ and $\gamma_{|(t_i,t_{i+1})}$, $i=0,\ldots, k-1$, is a smooth curve in the regular part with $L(\gamma)<d(x_0,x_1)+\varepsilon$. Since $\phi$ is a Riemannian isometry on the regular part, it follows that $\phi: \oB_r(x) \To \oB_r(\bar x)$ is $1$-Lipschitz. The same argument applied to $\phi^{-1}$ shows that $\phi$ is an isometry as claimed.
\end{proof}

Hence, a Riemannian orbifold in the sense of Definition \ref{dfn:orbifold} admits a smooth orbifold structure and a compatible Riemannian structure that in turn induces the metric structure. Conversely, every paracompact smooth orbifold admits a compatible Riemannian structure that turns it into a Riemannian orbifold \cite[Ch.~III.1]{MR1744486}. In this sense the two definitions are equivalent.

\begin{rem}\label{rem:clarify_citation}
Alternatively, one can argue by induction on the dimension to show that two isometric actions of finite groups $G_1$ and $G_2$ on simply connected Riemannian manifolds $M_1$ and $M_2$, respectively, are conjugated if the corresponding quotient spaces are isometric. Using the induction assumption one shows as in the proof of lemma \ref{lem:metric_det_group} that an isometry between $M_1/G_1$ and $M_2/G_2$ is a covering of Riemannian orbifolds in the sense of Definition \ref{dfn:riem_orb_covering}. Then the claim follows from Theorem \ref{thm:existence_universal_orbifold}.
\end{rem}

\subsection{Orbifold coverings}\label{sub:orbifold_coverings}

The concept of a covering orbifold was introduced by Thurston \cite[Def.~13.2.2]{Thurston}. In the Riemannian setting his definition reads as follows.

\begin{dfn}[Thurston] \label{dfn:riem_orb_covering} A \emph{covering orbifold} of a Riemannian orbifold $\Orb$ is a Riemannian orbifold $\hat\Orb$ together with a surjective map $p : \hat\Orb \To \Orb$ that satisfies the following property. Each point $x\in \Orb$ has a neighborhood $U$ isometric to some $M/G$, as in Definition \ref{dfn:orbifold}, for which each connected component $U_i$ of $p^{-1}(U)$ is isometric to $M/G_i$ for some subgroup $G_i<G$ such that the following diagrams commute
\[
	\begin{xy}
		\xymatrix
		{
		   U_i \ar[r]^{\sim \; \;} \ar[d]_{p}  & M/G_i \ar[d]    \\
		   U \ar[r]^{\sim \; \;}  & M/G   
		}
	\end{xy}
\]
We refer to the map $p : \hat\Orb \To \Orb$ in the definition as an \emph{orbifold covering} or a \emph{covering of Riemannian orbifolds}. 
\end{dfn}

\begin{figure}
	\centering
		\def\svgwidth{1.0\textwidth}
		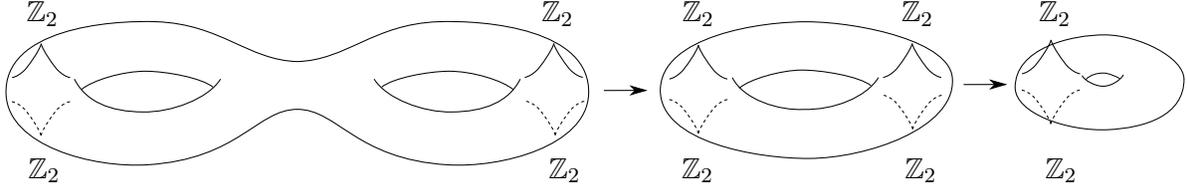
	\caption{Composition of two two-fold orbifold coverings.}
	\label{fig:covering_example}
\end{figure}

\begin{rem} If $p:\hat \Orb \To \Orb$ is a covering of smooth orbifolds in the sense of \cite[Def.~13.2.2]{Thurston}, then any Riemannian orbifold metric on $\Orb$ lifts to a Riemannian orbifold metric on $\hat \Orb$ with respect to which the covering is a covering of Riemannian orbifolds in the above sense.
\end{rem}

\begin{rem} The $G_i$ appearing in Definition \ref{dfn:riem_orb_covering} do not have to be isomorphic. For instance, consider the $4$-fold orbifold covering sketched in Figure \ref{dfn:riem_orb_covering}. The singular points with local group $\Z_2$ on the base orbifold have three preimages. Two of them are singular with local group $\Z_2$ and the third one is regular. This can only happen if the covering is not \emph{Galois} in the sense of the following definition.
\end{rem}

\begin{dfn}\label{dfn:riem_orb_covering_deck} The \emph{deck transformation group} of a covering $p : \hat\Orb \To \Orb$ of Riemannian orbifolds is defined as the group of all isometries of $\hat \Orb$ that leave the fibers of $p$ invariant. The covering $p$ is called \emph{Galois} if the deck transformation group acts transitively on the fibers of $p$.
\end{dfn}
Let us record the following statement.
\begin{lem}\label{lem:regular_cover_quotient_metric} Suppose $p : \hat\Orb \To \Orb$ is a Galois covering of Riemannian orbifolds. Then the metric on $\Orb$ coincides with the quotient metric with respect to the action of the deck transformation group $G$ on $\hat \Orb$.
\end{lem}
\begin{proof} Since $\hat \Orb$ is a length space, the distance between two $G$-orbits can be realized as the infimum of paths connecting the two orbits. Moreover, since the map $p$ is $1$-Lipschitz it follows that the metric of $\Orb$ is majorized by the quotient metric. The converse follows since $\Orb$ is a length space and since e.g. paths in $\Orb$ which are piecewise minimizing can be arc-length-preservingly lifted to $\hat \Orb$.
\end{proof}

In the following we sometimes omit the term \emph{Riemannian} if a property of a Riemannian orbifold is actually a property of the underlying smooth orbifold. An orbifold covering $p : \hat\Orb \To \Orb$ is called \emph{universal} if, given a choice of points $\hat x_0 \in \hat \Orb$ and $x_0 \in \Orb$ with $\varphi(\hat x_0)=x_0$, for any orbifold covering $p' : \Orb' \To \Orb$ and a base point $x_0'$ with $\varphi'(x_0')=x_0$, there exists an orbifold covering $q:\hat \Orb \To \Orb'$ with $q(\hat x_0)=x_0'$ and $p=p'\circ q$. The following statement is due to Thurston.

\begin{thm}[Thurston]\label{thm:existence_universal_orbifold} For every orbifold $\Orb$ there exists a universal orbifold covering $p : \hat\Orb \To \Orb$, which is Galois. Moreover, if $M$ is a simply connected manifold and $p : M \To \Orb$ is a covering of orbifolds, then it is a universal covering.
\end{thm}

There are several ways of proving Theorem \ref{thm:existence_universal_orbifold}. For instance, in \cite{Thurston} Thurston describes a fiber product for orbifolds and uses it to obtain a universal covering orbifold as an inverse limit. In \cite{MR2060653} a proof using the notion of ``orbifold loops'' similarly to ordinary loops in case of topological spaces is sketched. Alternatively, a universal covering of a Riemannian orbifold $\Orb$ can be constructed as follows. Consider the orthonormal frame bundle $F \Orb$, which is a manifold, and let $\tilde{F}$ be its universal covering. Let $\sim$ be the equivalence relation on $\tilde{F}$ which identifies two points when they lie in the same connected component of the preimage of a fiber of $F \Orb \To \Orb$. Then the induced map $\tilde{F}/\sim \To \Orb$ is a universal covering of Riemannian orbifolds. The details of this line of reasoning are worked out in an earlier arXiv version of this paper.

As pointed out in the introduction, in the noncomplete case a covering in the sense of Thurston does not need to be a submetry. In fact, this property already fails for ordinary coverings of Riemannian manifolds as the following example illustrates.

\begin{exl}\label{exl:covering_but_no_sub} Consider the subset $X:=(\R^2- (\{0\} \times \R) \cup (\{0\}\times \bigcup_{n=1}^{n=\infty} (1/10^n,2/10^n))$ of $\R^2$ with its restricted Riemannian metric. Let $\tilde X$ be the universal cover of $X$ with the lifted Riemannian metric for which the covering $p:\tilde X \To X$ becomes a local isometry. Endow $X$ and $\tilde X$ with the induced length metrics. With respect to this metric the distance between the points $x=(-1,0)$ and $y=(1,0)$ in $X$ is $2$. Now let $\tilde x$ and $\tilde y$ be lifts of $x$ and $y$ in $\tilde X$. The point $\tilde y$ is the endpoint of a lift $\tilde \gamma$ of a path $\gamma:[0,1]\To X$ with $\gamma(0)=x$ and $\gamma(1)=y$. Moreover, any path between $\tilde x$ and $\tilde y$ projects to a path on $X$ that is homotopic to $\gamma$ with its endpoints fixed. Since the distance between $\tilde x$ and $\tilde y$ is the infimum of the length of all curves between $\tilde x$ and $\tilde y$, it follows that this distance is strictly larger than $2$. In particular, the point $y$ is not contained in the ball $p(\cB_2(\tilde x))$ and so $p$ is not a submetry.
\end{exl}

However, using the fact that covering maps have the curve lifting property, one can show that the covering constructed in Example \ref{exl:covering_but_no_sub} is a weak submetry. Here a map $f:X \To Y$ between metric spaces $X$ and $Y$ is called \emph{weak submetry} if $p(\oB_r(x))=\oB_r(p(x))$ holds for all $x\in X$ and all $r\geq 0$. A submetry is always a weak submetry. The converse is true if the space $X$ is proper, i.e. its bounded closed balls are compact. More generally, in Lemma \ref{lem:weak_submetry_implied} we show that a covering of Riemannian orbifolds is a weak submetry. We could not decide yet whether the converse is always true or not, even in the manifold case.

\begin{qst} Is a weak submetry between Riemannian manifolds (orbifolds) of the same dimension always a covering map?
\end{qst}

\subsection{Some more useful terminologies}\label{sub:terminologies}

Let us recall some more terminology and basic facts that will be needed in the sequel. For a point $x$ on a Riemannian orbifold $\Orb$ choose a neighborhood isometric to $M/G$ as in Definition \ref{dfn:orbifold} such that $G$ fixes a preimage $\bar x$ of $x$ in  $M$. It follows from Lemma \ref{lem:metric_det_group} that the linearized orthonormal action of $G$ on $T_{\bar x} M$ is uniquely determined up to conjugation by isometries. A representative of the corresponding conjugacy class in $\Or(n)$ is denoted as $G_x$ and is called \emph{local group} of $\Orb$ at $x$. The metric quotient $(T_{\bar x} M)/G$ is the tangent space $T_x \Orb$ of $\Orb$ at $x$. It coincides with the tangent cone of $\Orb$ at $x$ in the sense of metric geometry (cf. \cite{MR2117451}). The exponential map on $T_{\bar x} M$ is $G$-equivariant and descends to an exponential map $\exp_x$ defined on an open neighborhood of the origin $x_o$ of $T_x \Orb$.

Points with trivial local group are called \emph{regular}. The set of regular points is dense in $\Orb$. All other points are called \emph{singular}. If $\Orb$ is a global quotient $M/\Gamma$ of a Riemannian manifold by a (possible infinite) discrete group $\Gamma$, then we call points on $M$ regular or singular with respect to the projection $M\To M/\Gamma$ if they project to regular or singular points on $\Orb$, respectively.

Recall that a ball $\cB_r(x)$ in a Riemannian manifold $M$ is called \emph{totally convex} if it is normal and, in addition, every pair of points in $\cB_r(x)$ can be connected by a unique minimizing geodesic in $M$ that is contained in $\cB_r(x)$. Totally convex balls always exist \cite[Prop.~4.2]{MR1138207}. We call balls $B_r(\bar x)$ and $U_r(\bar x)$ in a Riemannian orbifold $\Orb$ \emph{normal} or \emph{totally convex}, respectively, if it is contained in a neighborhood isometric to $M/G$ as in Definition \ref{dfn:orbifold} where $G$ fixes a preimage $x$ of $\bar x$ in $M$, and the ball $\cB_r(x)$ is normal or totally convex, respectively. If $B_r(\bar x)$ is normal then any orbifold geodesic starting at $\bar x$ exists and is distance minimizing up to length $r$. If $\cB_r(\bar x)$ is totally convex in the above sense, then $\cB_r(\bar x)$ is also totally convex in the sense that any pair of points in $\cB_r(\bar x)$ can be connected by a distance minimizing geodesic and that any such distance minimizing geodesic in $\Orb$ is contained in $\cB_r(\bar x)$.

\section{Metric characterization of orbifold coverings}\label{sec:metric_cover}

In this section we prove Theorem \ref{thm:covering_prop}. Recall that given a metric space $(Y,d)$ we call its restriction to a subspace $Z\subset Y$ the \emph{restricted metric} on $Y$, and the metric on $Y$ which measures the infimum of the lengths of paths between two points in $Y$ the \emph{induced length metric}. Also recall that the restriction of a submetry $f:X\To Y$ to the preimage $f^{-1}(Z)$ of any subset $Z\subseteq Y$ is a submetry with respect to the restricted metrics, and that the composition of two submetries is again a submetry. This implies that a map $p : \Orb \To Y$ satisfies condition $(ii)$ of Theorem \ref{thm:covering_prop} if it is a \emph{discrete} submetry, i.e. a submetry with discrete fibers. It is also easy to see that condition $(ii)$ of Theorem \ref{thm:covering_prop} is satisfied if $Y$ is an $n$-dimensional Riemannian orbifold and $p : \Orb \To Y$ is a covering of orbifolds. We will show the converse by induction on the dimension. Moreover, at the end of this section we show that both conditions imply that the map $p$ is a weak submetry (see Lemma \ref{lem:weak_submetry_implied}), and  a submetry if $\Orb$ is complete. This will complete the proof of Theorem \ref{thm:covering_prop}.

To prove that the map $p:\Orb \To Y$ is a covering of Riemannian orbifolds under the assumptions of Theorem \ref{thm:covering_prop}, $(ii)$, we draw on results about submetries between Alexandrov spaces from Lytchak's thesis \cite{MR1938523}. For the definition and background on Alexandrov spaces we refer the reader to e.g. \cite{MR1185284,MR1835418}. We will simply use that a Riemannian orbifold, although it is in general neither complete nor does it have a lower curvature bound, is locally an Alexandrov space, and that the range of a submetry inherits this property from its domain \cite[Prop.~4.4]{MR1938523}. Moreover, let us mention that the tangent space at a point $x$ in an Alexandrov space $X$ can be defined as the cone over the completion of the space of unit speed geodesics starting at $x$ endowed with a certain ``angle'' metric.
In \cite{MR2117451,MR1938523} a notion of differentiability is introduced for Lipschitz maps between metric spaces. Using this notion it is shown in \cite[Prop.~5.1]{MR1938523} that for any submetry $f:X\To Y$ between Alexandrov spaces $X$ and $Y$ and any $x\in X$ there exists a homogenous submetry $Df_x: T_xX \To T_{f(x)}Y$ (the differential of $f$ at $x$) with the property that if $\gamma$ and $p\circ\gamma$ are geodesics starting at $x$ and $f(x)$, respectively, then $Df_x$ maps $\gamma'(0)$ to $(p\circ\gamma)'(0)$. Moreover, in \cite[Prop.~9.1]{MR1938523} it is shown that the submetry $f$ has discrete fibers if and only if $X$ and $Y$ have the same dimension.

In order to prove that $p:\Orb\To Y$ is a covering of Riemannian orbifolds under the assumptions stated in Theorem \ref{thm:covering_prop}, $(ii)$, we particularly have to show that for any $y \in Y$ and any $x\in p^{-1}(y)$ there exists some $R>0$ such that the restriction $p:U_R(x) \To U_R(y)$ is of the form $M/H \To M/G$ as in Definition \ref{dfn:riem_orb_covering}. Once this is proven the claim immediately follows in the case where $p$ has finite fibers. Otherwise we additionally have to show that the radius $R$ can be chosen independently of the point $x$ in the fiber.

We first deal with the local problem. We need the following lemma.

\begin{lem}\label{lem:normal} Let $y\in Y$ and $x\in p^{-1}(y)$. Suppose that the ball $B_R(x)$ is normal and that the distance of $x$ to any other element in $p^{-1}(y)$ is at least $4R$. Then $p: B_R(x)\To B_R(y)$ is a submetry and $p$ maps $S_r(x)$ onto $S_r(y)$ for $r<R$. Moreover, geodesics starting at $y$ exist and are minimizing up to length $R$, and each point in $U_R(y)$ can be connected to $y$ by a unique minimizing geodesic.
\end{lem}
\begin{proof} Since $p$ is a submetry, our assumption on $R$ implies that $p^{-1}(B_R(y))\cap B_{2R}(x) = B_{R}(x)$. It follows that the restriction $p: B_R(x)\To B_R(y)$ is a submetry. Moreover, the submetry property applied at $x$ and at points in $S_r(x)$ shows that for $r<R$ the distance $r$-sphere $S_r(x)$ is mapped onto $S_r(p(x))$.

To prove the second part of the lemma we note that a possibly smaller ball $B_{R'}(x)$ is totally convex so that $p$ restricts to a submetry $p: B_{R'}(x)\To B_{R'}(y)$ between Alexandrov spaces. In particular, we can lift initial directions at $y$ to initial directions at $x$ via the homogeneous submetry $Df_x: T_x \Orb \To T_{f(x)}Y$. Since the ball $B_R(x)$ is normal by assumption, geodesics starting at $x$ exist and are minimizing up to length $R$. Since $S_r(x)$ is mapped onto $S_r(p(x))$ for $r<R$, their images in $B_R(y)$ are minimizing up to length $R$ as well. Hence, in each initial direction at $y$ there exists a minimizing geodesic up to length $R$. Since geodesics in Alexandrov spaces do not branch, all geodesics starting at $y$ are of this type up to length $R$. Moreover, since a point $z$ in $U_R(y)$ can be lifted to a point in $U_R(x)$, it lies on a minimizing geodesic starting at $y$. Finally, the existence of two distinct such minimizing geodesics would imply a contradiction to the minimizing property up to length $R$, since one would find shortcuts in this case nearby $z$.
\end{proof}

We also need the following statement which may be of independent interest.

\begin{lem}\label{lem:restricted_submetry} Let $f:X \To Y$ be a submetry between metric spaces and suppose that $X$ is proper, i.e. all its closed bounded subsets are compact. Then for any closed subset $Z\subset Y$ the restriction of $f$ to $f^{-1}(Z)$ is a submetry with respect to the induced length metrics on $f^{-1}(Z)$ and on $Z$.
\end{lem}
\begin{proof} Clearly, the restriction of $f$ to $f^{-1}(Z)$ is $1$-Lipschitz with respect to the induced length metrics. It remains to show that for any $x\in f^{-1}(Z)$, any $r>0$ and any $y\in B_r(f(x))\subset Z$ there exists some $x' \in B_r(x)\subset f^{-1}(Z)$ with $f(x')=y$. Here we mean the balls with respect to the induced length metrics. For such a $y$ and any $n >0$ there exists some $1$-Lipschitz curve $\gamma_n: [0,r+1/n] \To Z$ with $\gamma(0)=f(x)$ and $\gamma(r+1/n)=y$. Since $X$ is proper by assumption there exists a $1$-Lipschitz curve $\bar{\gamma}_n: [0,r+1/n] \To B_{r+1/n}(x)\subset f^{-1}(Z)$ with $f\circ \bar{\gamma}_n = \gamma_n$ by \cite[Lem.~4.4]{MR2167848}. Let $x'_n \in B_{r+1/n}(x)$ be the endpoint of the curve $\bar{\gamma}_n$. By properness of $X$ and the fact that $f^{-1}(Z)\subset X$ is closed, a subsequence of $(x'_n)$ converges to some $x' \in B_r(x) \subset f^{-1}(Z)$, with $f(x')=y$ by continuity.
\end{proof}

Now we can settle the local problem.
\begin{prp}\label{prp:orb_covering_char} Suppose $p:\Orb \To Y$ has the properties stated in Theorem \ref{thm:covering_prop}, $(ii)$, and that $p$ has finite fibers. Then $p$ is a covering of Riemannian orbifolds in the sense of Definition \ref{dfn:riem_orb_covering}.
\end{prp}
\begin{proof}
The proof is by induction on the dimension $n$ of $\Orb$. We will prove the base case $n=1$ of the induction along the way. As pointed out above it suffices to show that for any $y \in Y$ and any $x\in p^{-1}(y)$ there exists some $R>0$ such that the restriction $p:U_R(x) \To U_R(y)$ is of the form $M/H \To M/G$ as in Definition \ref{dfn:riem_orb_covering}. Fixing such $x$ and $y$, by Lemma \ref{lem:group_dete_deck_trans} we can assume that $\Orb$ is a manifold $M$ and that $p$ is a submetry. We choose $R$  so that the ball $B_R(x)$ is normal and so that the distance of $x$ to any other element in $p^{-1}(y)$ is at least $4R$. Then by Lemma \ref{lem:normal} the restriction $p:B_R(x) \To B_R(p(x))$ is a submetry and for $0<r<R$ the map $p$ restricts to $p: S_{r}(x)\To S_{r}(p(x))$. For $n=1$ the space of directions at $y$, and thus all the $S_{r}(p(x))$, consists either of two points or of one point. In the first case the restriction $p:U_R(x) \To U_R(y)$ is an isometry, and in the  second case it is the quotient map for an isometric reflection at $x$. In particular, the base of our induction argument holds.

Suppose that $n>1$ and that the induction assumption holds in all lower dimensions. Since $B_R(x)$ is proper, the restrictions $p: S_{r}(x)\To S_{r}(p(x))$ are submetries with respect to the induced intrinsic length metrics by Lemma \ref{lem:restricted_submetry}, and thus coverings of Riemannian orbifolds by our induction assumption. We claim that there exists a finite group $G$ which acts isometrically on $S_{r}(x)$ so that $p$ is the quotient map for this action. For $n>2$ this follows immediately from Theorem \ref{thm:existence_universal_orbifold} and Lemma \ref{lem:regular_cover_quotient_metric} since $S_{r}(x)$ is a simply connected manifold in this case. For $n=2$ one can e.g. observe that the orbifold fundamental group of $S_{r}(x)$ is a normal subgroup of $S_{r}(p(x))$ so that the covering $p: S_{r}(x)\To S_{r}(p(x))$ is Galois, or carry out the construction by hand. Since $B_R(x)$ is normal, we can extend the action of $G$ via the exponential map to a continuous action on $B_R(x)$. By Lemma \ref{lem:normal} the group $G$ and the thus obtained action on $B_R(x)$ does not depend on the specific value of $r<R$ we started with. 

We claim that $G$ acts isometrically on $B_R(x)$. It suffices to show that $G$ acts isometrically on $B_R(x) \backslash \{x\}$. Since $G$ acts isometrically on each $S_r(x)$, $0<r<R$, and preserves distances in radial directions, the claim follows if we can show that the action of $G$ on $B_R(x) \backslash \{x\}$ is smooth. By the same reason the action restricted to a distance sphere $S_r(x)$ is smooth. We conjugate the action of $G$ with the smooth exponential map to obtain an action on $T_x M$. Then in spherical coordinates the map induced by an element $g\in G$ is of the form $S^{n-1}\times \R_{>0} \To S^{n-1}\times \R_{>0}$, $(x,r)\mapsto (g(x),r)$, which is clearly smooth. Hence, $p:U_R(x) \To U_R(y)$ is the quotient map for the isometric action of $G$ on $U_R(x)$. This completes the proof of the proposition.
\end{proof}

Now we drop the assumption that $p:\Orb \To Y$ has finite fibers, i.e. we only assume that $p$ has discrete fibers, is onto, locally $1$-Lipschitz and that each point $y \in Y$ has a neighborhood $U$ such that the restriction of $p$ to $p^{-1}(U)$ is a submetry with respect to the restricted metrics as stated in Theorem \ref{thm:covering_prop}. By Proposition \ref{prp:orb_covering_char} we can assume that $Y$ is a Riemannian orbifold $\Orb'$. Moreover, by its proof, in order to prove that $p$ is still a covering of Riemannian orbifolds, it suffices to show that for any $y \in \Orb'$ there exists some $r>0$ such that for any $x \in p^{-1}(y)$ the balls $B_r(y)$ and $B_r(x)$ are normal. For, in this case, perhaps after further decreasing $r$, for any $x\in p^{-1}(y)$ the points $x$ and $y$ satisfy the assumptions of Lemma \ref{lem:normal}. This will be proved in Lemma \ref{lem:sub_geo_to_geo_2}.

\begin{lem}\label{lem:sub_geo_to_geo_local} Let $y\in \Orb'$ and $x \in \Orb$ with $p(x)=y$. There exists some $r>0$ for which any distance minimizing geodesic $\bar \gamma:[0,r] \To \Orb'$ with $\bar \gamma(0)=y$ can be lifted to a distance minimizing geodesic $\gamma$ on $\Orb$ with $\gamma(0)=x$. Moreover, given any $v\in T_x \Orb$ with $D_x p(v) = \bar \gamma'(0)$ the lift $\gamma$ can be chosen such that $\gamma'(0)=v$.
\end{lem}
\begin{proof} We choose $R>0$ such that the ball $\cB_R(x) \subset \Orb$ is totally convex and such that the assumptions of Lemma \ref{lem:normal} are satisfied. In particular, for all $r\leq R$ the restrictions $p:\cB_r(x) \To \cB_r(p(x))$ are submetries between Alexandrov spaces of curvature bounded from below in this case. Then for any geodesic $\bar \gamma : [0,r] \To \Orb$ with $\bar\gamma(0)=y$, and any $v \in T_x \Orb$ with $D_x p(v) = \bar \gamma'(0)$, there exists a lift $\gamma$ of $\bar \gamma$ to $\Orb'$ with $\gamma'(0)=v$ by \cite[Lem.~5.4]{MR1938523}. By the submetry property also the lift $\gamma$ of $\bar \gamma$ is distance minimizing as claimed.
\end{proof}

 We will need the following completeness property, which is immediate if $\Orb$ is complete.

\begin{lem}\label{lem:sub_geo_to_geo_ext_op} Suppose the restriction of a distance minimizing geodesic $\bar \gamma:[0,s] \To \Orb'$ to $[0,s)$ lifts to a distance minimizing geodesic $\gamma: [0,s) \To \Orb$. Then $\gamma$ can be extended to a lift of $\bar \gamma$ on $[0,s]$.
\end{lem}
\begin{proof}Let $r>0$ such that the restriction $p_X$ of $p$ to $X:=p^{-1}(\cB_r(\bar \gamma(s))$ is a submetry. We can assume that $\gamma$ is contained in $X$. Set $x=\gamma(0)$ and $y=p(x)$. Since $p_X$ is a submetry and $\bar \gamma$ is distance minimizing, there exists some $z\in X$ with $d(\gamma(s/2),z)=s/2$ and $p(z)=\bar{\gamma}(s)$. By the same reason we have $d(x,z)=s$. We claim that $z$ is the limit of $\gamma(t)$ as $t$ tends to $s$. Suppose this is not the case. Then there exists some $s'\in [s/2,s)$ maximal with the property that $d(\gamma(s/2),z)=d(\gamma(s/2),\gamma(s'))+d(\gamma(s'),z)$. Since $\Orb$ is a length space, for some small normal ball $\cB_{\varepsilon}(\gamma(s'))\subseteq X$, $\varepsilon < s-s'$, we have that $d(\gamma(s'),z)=\mathrm{inf}_{u \in S_{\varepsilon}(\gamma(s'))} (d(\gamma(s'),u)+d(u,z))$ where $S_{\varepsilon}(\gamma(s'))$ is the distance $\varepsilon$-sphere at $\gamma(s')$. By compactness the infimum is attained, say at $u_0 \in S_{\varepsilon}(\gamma(s'))$. By maximality of $s'$ the point $u_0$ does not lie on $\gamma$. But this implies $d(x,z)<s$ since we can cut short at $\gamma(s')$, a contradiction. It follows that $\gamma$ can be extended to a lift of $\bar \gamma$.
\end{proof}

Now we can strengthen the statement of Lemma \ref{lem:sub_geo_to_geo_local}.

\begin{lem}\label{lem:sub_geo_to_geo_ext_cl} Let $y\in \Orb'$ and fix $r>0$ such that the restriction of $p$ to $p^{-1}(\cB_r(y))$ is a submetry. Then the conclusion of Lemma \ref{lem:sub_geo_to_geo_local} holds for any $x\in p^{-1}(y)$ with this $r$. 
\end{lem}
\begin{proof} 
Let $\bar \gamma:[0,r] \To \Orb'$ be a distance minimizing geodesic with $\bar \gamma(0)=y$, and let $v\in T_x \Orb$ with $D_x p(v) = \bar \gamma '(0)$. The set of all $s\in [0,1]$ for which there exists a lift $\gamma:[0,s]\To \Orb$ of $\bar \gamma_{|[0,s]}$ with $\gamma(0)=x$ and $\gamma'(x)=v$ is closed by Lemma \ref{lem:sub_geo_to_geo_ext_op} and open in $[0,r]$ by Lemma \ref{lem:sub_geo_to_geo_local}. Hence, the lift $\gamma$ exists on $[0,r]$. Since the restriction of $p$ to $p^{-1}(B_r(y))\supseteq B_r(x)$ is a submetry and $\bar \gamma$ is distance minimizing, the path $\gamma$ is distance minimizing, too.
\end{proof}

As a consequence, we can prove the following lemma.

\begin{lem}\label{lem:sub_pre_normal} Let $\cB_{r}(y)\subset \Orb'$ be a normal ball so that the restriction of $p$ to $p^{-1}(\cB_{r}(y))$ is a submetry. Then for any $x\in p^{-1}(y)$ geodesics starting at $x$ exist and minimize up to length $r$. 
\end{lem}
\begin{proof} Since $B_{r}(y)$ is normal, for any $x\in p^{-1}(y)$ and any $v\in T_x \Orb$ there exists a minimizing geodesic $\bar \gamma: [0,r] \To \Orb'$ with $\gamma(0)=y$ and $\gamma'(0)=D_x p(v)$. By Lemma \ref{lem:sub_geo_to_geo_ext_cl} this geodesic can be lifted to a minimizing geodesic $\gamma: [0,r] \To \Orb$ with $\gamma(0)=x$ and $\gamma(0)=v$. Hence, by uniqueness of geodesics with given initial conditions, geodesics starting at $x$ exist and minimize up to length $r$.

\end{proof}

We can apply Lemma \ref{lem:sub_pre_normal} to submetries between Riemannian orbifolds as follows. 

\begin{lem}\label{lem:sub_geo_to_geo_2} Let $\cB_{2r}(y)\subset \Orb'$ be a normal ball so that the restriction of $p$ to $p^{-1}(\cB_{2r}(y))$ is a submetry. 
Then for any $x\in p^{-1}(y)$ the ball $B_r(x)$ is normal.
\end{lem}
\begin{proof} Pick some $x\in p^{-1}(y)$. By Lemma \ref{lem:sub_pre_normal} the exponential map is nonsingular on $U_{2r}([0]) \subset T_{x} \Orb=\R^n/H_x$. 
Thus we can pull-back the metric from $\Orb$ to a Riemannian metric on $U_{2r}(0)\subset \R^n$ with respect to which the ball $B_r(0)$ is normal. We claim that $H_{x}$ acts isometrically on $U_{2r}(0)$ with respect to this metric, and that it induces an isometry between $U_{2r}(0)/H_x$ and $U_{2r}(x)$ with respect to the induced length metrics. Since $H_{x}$ leaves the fibers of the exponential map $U_{2r}(0)\To \Orb$ invariant, it acts isometrically on the regular part in $U_{2r}(0)$ which is open and dense. Now our claim follows by similarly as in the last part of the proof of Lemma \ref{lem:metric_det_group}. Hence, the ball $B_r(x)$ is normal as claimed.
\end{proof}

By what has been said after Proposition \ref{prp:orb_covering_char} and by the preceding lemma the map $p:\Orb \To Y$ is a covering of Riemannian orbifolds. In order to complete the proof of Theorem \ref{thm:covering_prop} it remains to prove the following lemma.

\begin{lem}\label{lem:weak_submetry_implied} A map $p: \Orb \To \Orb'$ that satisfies conditions $(i)$ and $(ii)$ in Theorem \ref{thm:covering_prop} is a weak submetry. Moreover, it is a submetry if $\Orb$ is complete.
\end{lem}
\begin{proof} We have already observed that the map $p$ is $1$-Lipschitz as a locally $1$-Lipschitz map between length spaces. To recognize it as a weak submetry we have to show that for any points $x\in \Orb$ and $x',y' \in \Orb'$ with $p(x)=x'$ and $d(x',y')=r$,  and any $\varepsilon>0$ there exists a point $y \in \Orb$ with $p(y)=y'$ and $d(x,y)< r+ \varepsilon$. 

So let $\varepsilon >0$ and let $\bar \gamma: [0,r'] \To \Orb'$ be a $1$-Lipschitz path with $\bar \gamma(0)=x'$, $\bar \gamma(r')=y'$ and length $L(\bar \gamma)<d(x',y')+ \varepsilon$. By compactness there exists a finite subdivision $0=t_0<t_1<\cdots<t_k=r$ and radii $r_i$, $i=0,\ldots,k$, such that each ball $B_{r_i}(\bar \gamma(t_i))$ is totally convex and satisfies the condition in Definition \ref{dfn:riem_orb_covering}, and such that $\bar \gamma([0,t_1]) \subseteq B_{r_0}(\bar \gamma(0))$, $\bar \gamma([t_{i-1},t_{i+1}]) \subseteq B_{r_i}(\bar \gamma(t_i))$, $i=1,\ldots,k-1$, and $\bar \gamma([t_{k-1},r']) \subseteq B_{r_k}(\bar \gamma(r'))$ holds. Therefore, we can assume that the restriction of $\bar  \gamma$ to each interval $[t_i,t_{i+1}]$, $i=0,\ldots,k-1$, is distance minimizing. By Lemma \ref{lem:sub_geo_to_geo_ext_cl} we can lift $\bar \gamma$ to a $1$-Lipschitz path $\gamma: [0,r'] \To \Orb$ with the same property and length $\leq d(x',y')+ \varepsilon$. The point $y=\gamma(r')$ has the desired properties and so the first claim follows.

If $\Orb'$ is complete, then its closed, bounded balls are compact by the Hopf-Rinow theorem for length spaces \cite[Thm.~2.5.28]{MR1835418}. In this case a subsequence of the points $y'(\varepsilon)$ in the argument above converge to a point $y'$ with $p(y')=y$ and $d(x',y')=r$. Hence, $p$ is a submetry in this case as claimed.
\end{proof}

\section{The metric double covering} \label{sec:double_covering}

The aim of this section is to prove Proposition \ref{prp:metric_double}. Let us explain the required notions first. The \emph{metric double} of a metric space $X$ along a closed subspace $Y\subset X$ is defined as the quotient of the disjoint union of two copies $X_1$ and $X_2$ of $X$ by the equivalence relation that identifies copies of elements $y\in Y$ in $X_1$ and $X_2$, endowed with the unique maximal metric that is majorized by the metrics on the two copies $X_1$ and $X_2$ of $X$ in $2_Y X$ \cite[3.1.24]{MR1835418}. Equivalently, this metric can be described as a \emph{gluing metric} where the distance between two points $x,y \in 2_Y X$ is defined as the infimum of
\[
			\sum_{i=0}^{k} d(x_i,y_i)
\] 
over all sequences $x_i,y_i$, $i=0,\ldots,k$ with $x=x_0$, $y=y_k$, $x_i,y_i \in X_j$, for $j=1$ or $j=2$, and $[x_i]=[y_{i+1}]$ in $2_Y X$, $i=0,\ldots,k-1$ \cite[3.1.12, 3.1.27]{MR1835418}. If $X$ is in addition a length space, then so is $2_Y X$ \cite[3.1.24]{MR1835418}.

We apply this construction in case of a Riemannian orbifold $X=\Orb$ and $Y$ being the closure of its codimension $1$ stratum $\Sigma_1 \Orb$. Recall that $\Sigma_1 \Orb$ is the set of points in $\Orb$ whose local group $G_x$ is generated by a single reflection. In the following we refer to this closure as the \emph{boundary} of $\Orb$ and denote it as $\partial \Orb$ since it coincides with the boundary of $\Orb$ in the sense of Alexandrov geometry (see \cite{MR1185284}) by the following lemma. 

\begin{lem}\label{lem:charact_codim_1_stratum}
A point $x$ on an orbifold $\Orb$ belongs to the closure of $\Sigma_1 \Orb$ if and only if its local group $G_x$ contains a reflection.
\end{lem}
\begin{proof} The claim can be reduced to the case in which $\Orb=\R^n/G_x$. In this case the only if direction follows by continuity. Conversely, suppose that $G:=G_x$ contains a reflection $g$. Since $G$ is finite and since the fixed point subspace of any nontrivial element in $G$ distinct from $g$ does not contain the one of $g$ by dimension reasons, it follows that there exists a sequence of points $x_i \in \R^n$ converging to $x$ such that $G_{x_i}=\left\langle g \right\rangle$. Hence, the images of the $x_i$ in $\Orb$ belong to $\Sigma_1 \Orb$ and converge to the image of $x$, i.e. $x$ is contained in the closure of $\Sigma_1 \Orb$.
\end{proof}

Recall the statement of Proposition \ref{prp:metric_double}: We want to prove that the metric double $2_{\partial \Orb} \Orb$ is a Riemannian orbifold and that the natural projection $2_{\partial \Orb} \Orb \To \Orb$ is a covering of Riemannian orbifolds. 

\begin{proof}[Proof of Proposition \ref{prp:metric_double}] The proof is by induction on the dimension. The claim is local and it is clear for points in $2\Orb = 2_{\partial \Orb} \Orb$ that do not project to $\partial \Orb$. Let $x \in 2 \Orb$ be a point that projects to $\partial \Orb$. According to Definition \ref{dfn:orbifold} we can assume that $\Orb$ is of the form $M/G$, the quotient of a Riemannian manifold ball by a finite group of isometries $G$ that fixes a preimage $\hat x$ of $x$ in $M$. In particular, for $n=1$ the underlying topological space of $\Orb$ is $[0,1)$ and in this case the claim is clear.

Fix some dimension $n>1$ and suppose that the claim is true in all lower dimensions. Let $S_{\hat x}$ be the unit sphere in the tangent space $T_{\hat x} M$. Then $S_{\hat x}/G$ is a Riemannian orbifold with $\R_{\geq 0}\cdot (\partial (S_{\hat x}/G)) = \partial (T_{\hat x} M/G)$. By induction assumption the metric double $2(S_{\hat x} x/G)$ is a covering orbifold of $S_{\hat x}/G$. Therefore, there exists an index $2$ subgroup $H$ of $G$ such that the identity on $S_{\hat x}/G$ lifts to an isometry $\theta: 2(S_{\hat x}/G) \To S_{\hat x}/H$. For $n=2$ this follows from the classification of finite subgroups of $\Or(2)$ and for $n>2$ it follows from Theorem \ref{thm:existence_universal_orbifold}. In particular, $\theta$ is equivariant with respect to the natural reflection of $2(S_{\hat x}/G)$ and the action of $G/H\cong \Z_2$ on $S_{\hat x}/H$.

Applying the exponential map yields an equivariant homeomorphism between small metric balls $\theta : 2\Orb \supset B_r(x) \To B_r(y) \subset M/H$ where $y$ is the coset of $\hat x$ in $M/H$. By construction this map has the property that $\theta(B_{r'}(x))=B_{r'}(y)$ holds for all $r' \in [0,r]$ and it descends to the identity map on $B_r(x) \subset \Orb$ which is an isometry. Since all metrics involved are length metrics, it follows that the map $\theta$ restricts to an equivariant isometry $\theta : 2\Orb \supset B_{r/9}(x) \To B_{r/9}(y) \subset M/H$ (for the elementary but tedious details we refer to Appendix \ref{appendix}). 
This shows that the projection from $2\Orb$ to $\Orb$ is a covering of Riemannian orbifolds and so the claim follows by induction.\end{proof}

\section{appendix}\label{appendix}

To fill in the details for our claim that the map $\theta : 2\Orb \supset B_{r/9}(x) \To B_{r/9}(y) \subset M/H$ in the proof of Proposition \ref{prp:metric_double} is an isometry, we first state some simple facts about the metric double of a metric space $(X,d)$ along a closed subspace $Y$. In the following we denote the two copies of $X$ sitting in $2_Y X$ by $X_1$ and $X_2$.

\begin{lem}\label{lem:connecting_path_double}
Let $\gamma:[0,1] \To 2_Y X$ be a path connecting points $\gamma(0) \in X_1$ and $\gamma(1) \in X_2$. Then the set $M=\{t\in [0,1]|\gamma(t) \in Y\}\subset [0,1]$ is a nonempty union of closed intervals.
\end{lem}
\begin{proof} The subspaces $X_1$ and $X_2$ are closed in $2_Y X$ since $Y$ is closed in $X$. Moreover, we have $2_Y X = X_1 \cup X_2$ and $Y=X_1 \cap X_2$. Hence, $[0,1]=\gamma^{-1}(X_1)\cup \gamma^{-1}(X_2)$ and so there is some $t \in [0,1]$ with $\gamma(t) \in X_1 \cap X_2 =Y$ since $[0,1]$ is connected. This shows that $M$ is non-empty. By continuity $M$ is also closed. Every closed subset of $[0,1]$ is a union of closed intervals.
\end{proof}

There is a natural reflection $s:2_Y X \To 2_Y X$ that interchanges the two copies $(X_1,d)$ and $(X_2,d)$ of $X$ in $(2_Y X,d_2)$ and fixes the subspace $Y=X_1\cap X_2\subset 2_Y X$ pointwise. The fact that it identifies $X_1$ and $X_2$ isometrically by definition implies that it is an isometry.

\begin{lem}\label{lem:double metric_estimate}
For two points $x,x' \in X_1$ we have $d(x,x')=d_2(x,x')\leq d_2(x,s(x'))$. In other words, the embedding $X_1 \hookrightarrow 2_Y X$ is isometric and the composition with the projection from $2_Y X$ to $2_Y X/s$ with the quotient metric is an isometry.
\end{lem}
\begin{proof} For any approximation $\sum_{i=0}^{k-1} d(x_i,x_{i+1})$ of $d_2(x,x')$ or $d_2(x,s(x'))$ we obtain the same approximation of $d(x,x')$ by mapping all the $x_i$ that lie in $X_2$ to $X_1$ via $s$. This shows $d_2(x,x')\leq d_2(x,s(x'))$, and $d(x,x')\leq d_2(x,x')$ by the triangle inequality. On the other hand, we have $d_2 \leq d$ by the characterization of $d_2$ as a maximal metric that is majorized by $d$ and so the claim follows.
\end{proof}

To prove that $\theta : 2\Orb \supset B_{r/9}(x) \To B_{r/9}(y) \subset M/H$ is an isometry we show that the metrics on $ B_{r/9}(x)$ and $B_{r/9}(y)$ satisfy the following properties. Let $s_G$ be the generator of $G/H$ acting on $M/H$. We denote the metrics on $2 \Orb$ and on $M/H$ by $d_2$ and $d_q$, respectively. In the subsequent lemma $(Z,z,d_Z,s_Z,\phi)$ may either be $(2\Orb, x,d_2,s, \mathrm{id})$ or $(M/H,y,d_q,s_G,\theta)$. 

\begin{lem}\label{lem:restricts_to_isometry_lem}
Let $(Z,d_Z)$ be a length space with an isometric involution $s_Z$. Suppose that there exists a homeomorphism $\phi: 2 \Orb \supset B_r(x)\To B_r(z)\subset Z$ with $\phi(x)=z\in Z_0:=\mathrm{Fix}(s_Z)$ that is $\Z_2$-equivariant with respect to the action of $s$ on $B_r(x)$ and the action of $s_Z$ on $B_r(0)$, and that is a radial isometry in the sense that $\phi(B_{r'}(x))=B_{r'}(z)$ holds for all $r'\in [0,r]$. Then the following holds true.
\begin{enumerate}
\item For $w,w' \in \phi(B_{r/3}(x)\cap X_1)$ we have
\[
		d_Z(w,w')\leq d_Z(w,s_Z(w')).
\]
\item For $w\in \phi(B_{r/3}(z)\cap X_1)$ and $w' \in \phi(B_{r/3}(z)\cap X_2)$ we have
\[
		d_Z(w,w')= \inf_{y'\in Z_0} \left( d_Z(w,y')+d_Z(y',w') \right).
\]
\end{enumerate}
\end{lem}
\begin{proof} $(i)$ Let $\gamma:[0,1] \To Z$ be a path connecting $w$ and $s_Z(w')$ whose length approximates $d_Z(w,s_Z(w'))$ up to some small $\varepsilon>0$. Because of $w,s_Z(w') \in B_{r/3}(z)$, we can assume that $\gamma$ is completely contained in $B_{r}(z)$. By Lemma \ref{lem:connecting_path_double} applied to $\phi^{-1}\circ \gamma$ there is some $t_0 \in [0,1]$ with $\gamma(t_0) \in \phi(Y\cap B_{r}(x)) \subset \mathrm{Fix}(s_Z)$. We define a new path $\tilde{\gamma}:[0,1]\To Z$ by $\tilde{\gamma}(t)=\gamma(t)$ for $t\in [0,t_0]$ and $\tilde{\gamma}(t)=s_Z(\gamma(t))$ for $t\in [t_0,1]$. The path $\tilde{\gamma}$ connects $w$ and $w'$ and has length $L(\tilde{\gamma})=L(\gamma)$ since $s$ is an isometry. Now the claim follows since $Z$ is a length space. 

$(ii)$ By the triangle inequality we have $d_Z(w,w') \leq\inf_{y'\in Z_0} \left( d_Z(w,y')+d_Z(y',w') \right)$. On the other hand, let $\gamma$ be a path connecting $w$ and $w'$ whose length approximates $d_Z(w,w')$. As above we can assume that $\gamma$ is completely contained in $B_{r}(z)$. Similarly as in $(i)$ we can use Lemma \ref{lem:connecting_path_double} to construct a path $\tilde{\gamma}$ connecting $w$ and $s_Z(w')$ that lies completely in $\theta(B_{r}(z)\cap X_1)$, intersects $Z_0$ and satisfies $L(\tilde{\gamma})=L(\gamma)$. Since $\tilde{\gamma}$ intersects $Z_0$, we have $L(\tilde{\gamma})\geq \inf_{y'\in Z_0} \left( d(w,y')+d(y',w') \right)$. The fact that $Z$ is a length space implies $d(w,w') \geq\inf_{y'\in Z_0} \left( d(w,y')+d(y',w') \right)$ and hence the claim follows.
\end{proof}

\begin{lem}\label{lem:restricts_to_isometry}
The map $\theta : 2\Orb \supset B_{r/9}(x) \To B_{r/9}(y) \subset M/H$ is an isometry as claimed in the proof of Proposition \ref{prp:metric_double}.
\end{lem}
\begin{proof} Let $z,z' \in B_{r/3}(x)\cap X_1$. By Lemma \ref{lem:double metric_estimate} we have $d(z,z')=d_2(z,z')=d(\overline{z},\overline{z}')$ where $\overline{z}$ and $\overline{z}'$ are the cosets of $z$ and $z'$ in $2 \Orb/s\cong \Orb$. Since $\theta$ descends to an isometry on $B_r(x)\subset \Orb$, the definition of the quotient metric on $M/G$ implies \[d_2(z,z')=\min \{d_q(\theta(z),\theta(z')),d_q(\theta(z),s_G(\theta(z')))\}.\] Now Lemma \ref{lem:restricts_to_isometry_lem}, $(i)$, shows that $d_2(z,z')=d_q(\theta(z),\theta(z'))$. By the same reason this identity holds for points $z,z' \in B_{r/3}(x)\cap X_2$.

Now let $z \in B_{r/9}(x)\cap X_1$ and $z' \in B_{r/9}(x)\cap X_2$. Then we have $d_2(z,z'),d_q(\theta(z),\theta(z'))\leq 2r/9$. Applying the first paragraph and Lemma \ref{lem:restricts_to_isometry_lem}, $(ii)$, in $B_r(x)$ and in $B_r(y)$ yields $d_2(z,z')=d_q(\theta(z),\theta(z'))$
which proves the claim.
\end{proof}

\end{document}